\documentclass[12point]{article}
\usepackage[all, ps, dvips]{xy} 
\usepackage{color} 
\usepackage{amsmath}
\usepackage{amsfonts}
\usepackage{amssymb}
\usepackage{mathrsfs}
\usepackage{amsthm}

\newtheorem{theorem}{Theorem}
\newtheorem{lemma}[theorem]{Lemma}

\everymath{\displaystyle}

\begin{document}
\title{Cohomology of Commuting Varieties of Connected Compact Reductive Lie Groups}
\author{Henry Scher}
\maketitle
\section{Abstract}
We calculate the rational cohomology of the commuting variety $X_{G, n}$ consisting of $n$-tuples of commuting elements of a compact reductive group $G$. This is done by studying a map from a related variety $Y_{G, n}$, which has easily calculated cohomology. The proof studies the fibers of the map and uses the Vietoris-Begle theorem to prove that the induced map on rational cohomology is an isomorphism. 

\section{Introduction}
Let $G$ be a connected compact reductive group. Then the $n$th commuting variety $X_{G, n}$ is the variety consisting of all $n$-tuples $(g_1, g_2, ..., g_n) \in G \times G$ that pairwise commute (i.e. $g_i g_j = g_j g_i$ for all $i, j$). Let $T$ be a maximal torus, and let $N_G(T)$ denote the normalizer of $T$ in $G$. Then let $Y_{G, n} := (G \times T^n)/N_G(T)$, where $N_G(T)$ acts by right-multiplication on $G$ and by conjugation on $T$. Let $f: Y_{G, n} \rightarrow X_{G, n}, f(g, t'_1, t'_2, ...t'_n ) = (g t'_1 g^{-1}, g t'_2 g^{-1}, ..., g t'_n g^{-1})$; note that $f$ is a $G$-equivariant map where $G$ acts on $Y_{G, n}$ by acting on the factor of $G$ by left-multiplication, and on $X_{G, n}$ by simultaneous conjugation. 

\begin{theorem}[Main Theorem]
The map $f$ induces an isomorphism on rational cohomology, i.e. $H^\cdot(X_{G, n}, \mathbb{Q}) \xrightarrow[f^*]{\sim} H^\cdot(Y_{G, n}, \mathbb{Q})$
\end{theorem}

This theorem is a generalization of two already-known theorems.
\begin{theorem} \label{n=0}
Let $G$ be a connected compact reductive group, $T$ a maximal torus, and $N_{G}(T)$ the normalizer of $T$ in $G$. Then $G/N_{G}(T)$ has trivial rational cohomology.
\end{theorem}

\begin{theorem}
Let $G$ be a connected compact reductive group, $T$ a maximal torus, and $N_{G}(T)$ the normalizer of $T$. Then $f: (G \times T)/N_G(T) \rightarrow G$ induces an isomorphism on rational cohomology.
\end{theorem}

Proofs of both of these can be found in [\ref{ARXIVPAPER}] (in the proof of Proposition 1). These can be seen as the $n = 0$ and $n = 1$ case of the main theorem, respectively.

This theorem allows relatively simple computation of the cohomology of the $n$th commuting variety. $X_{G, n}$ can be rewritten as $(G/T \times T^n)/W$, where $W$ is the Weyl group of $G$. The action of $W$ on $G/T$ is free, so the action of $W$ on $G/T \times T^n$ is free. Therefore, the cohomology of $X_{G, n}$ can be given by the $W$-invariants in the cohomology of $G/T \times T^n$. As the cohomology of $G/T$ is known (if the grading is ignored, it is the regular representation of $W$), and the cohomology of $T$ is isomorphic to the exterior algebra on the reflection representation of $W$), the cohomology is easy to calculate. 

\section{Proof of Main Theorem}
We rely on a theorem from algebraic topology to reduce the question to studying the fibers of $f$. 
\begin{theorem}[Vietoris-Begle Theorem] \label{Alg Top Theorem}

Let $f: Y \rightarrow X$ be a surjective map of compact metric spaces such that for all $x \in X$, $f^{-1}(x)$ is cohomologically trivial (with respect to some cohomology theory). Then $f$ induces an isomorphism on cohomology (for the same cohomology theory). 
\end{theorem}

As all elements of a compact group are diagonalizable, any commuting $n$-tuple is contained in some maximal torus. All maximal tori are conjugate, so $f$ is surjective. By Theorem \ref{Alg Top Theorem} using rational cohomology, we only need to prove the following lemma:
\begin{lemma}
For any commuting $n$-tuple $(g_1, g_2, ..., g_n)$, the set $\{(g, t'_1, t'_2, ..., t'_n) \in G \times T^n | \forall i \, g t'_i g^{-1} = g_i\}/N_G(T)$ has trivial rational cohomology. 
\end{lemma}

The rest of the paper will prove this lemma by rewriting this set until it is in a form known to have trivial rational cohomology. 

We can assume without loss of generality that the commuting $n$-tuple $(g_1, g_2, ..., g_n)$ is contained in our chosen maximal torus $T$. Change notation so that our commuting $n$-tuple is $(t_1, t_2, ..., t_n)$. Let $X = \{(g, t'_1, t'_2, ..., t'_n)| \forall i g t'_i g^{-1} = t_i\}$; then $f^{-1}(t_1, t_2, ..., t_n) = X/N_G(T)$. 

\begin{lemma} \label{Orbits lemma}
Let $G$ be a (not necessarily connected) reductive group. The $G$-orbit of an $n$-tuple of elements $(t_1, t_2, ..., t_n) \in T^n$ meets $T^n$ in exactly the $N_G(T)$-orbit of $(t_1, t_2, ..., t_n)$. In other words, if $g (t_1, t_2, ..., t_n) g^{-1} = (t'_1, t'_2, ..., t'_n) \in T^n$, then there is some $g' \in N_G(T)$ with $g' (t_1, t_2, ..., t_n) g'^{-1} = (t'_1, t'_2, ..., t'_n)$. 
\end{lemma}
\begin{proof}

We first reduce to the case that $G$ is connected. Let $g \in G$ such that $g (t_1, t_2, ..., t_n) g^{-1} = (t'_1, t'_2, ..., t'_n) \in T^n$. Let $T' = g T g^{-1}$. As all maximal tori are conjugate by an element of the connected component of the identity $G_0$, there is some $g_0 \in G_0$ such that $g_0 T g_0^{-1} = T'$. Then let $g_1 = g_0 g^{-1}$; an easy calculation shows that $g_1 \in N_G(T)$. As such, $g_1 t'_i g_1^{-1} \in T^n$, so let $t''_i = g_1 t'_i g_1^{-1}$. We then have that $g_0 (t_1, t_2, ..., t_n) g_0^{-1} = (t''_1, t''_2, ..., t''_n)$. If the theorem is true for connected $G$, then there is some $g_2 \in N_{G_0}(T)$ with $g_2 (t_1, t_2, ..., t_n) g_2^{-1} = (t''_1, t''_2, ..., t''_n)$. Let $g' = g_1^{-1} g_2$; then $g' (t_1, t_2, ..., t_n) g'^{-1} = g_1^{-1} (t''_1, t''_2, ..., t''_n) g_1 = (t'_1, t'_2, ..., t'_n)$. We therefore only need to prove this in the case that $G$ is connected. 

The $n = 1$ case is a consequence of Chevalley's theorem. We prove this in the $n = 2$ case; the general case is similar, and works by induction. The general strategy is to reduce to the case that $t'_i = t_i$ for $i > 1$ by the inductive assumption, and then to reduce to the $n = 1$ case for a subgroup of $G$. 

Assume $g (t_1, t_2) g^{-1} = (t'_1, t'_2) \in T^2$. Then $g t_2 g^{-1} = t'_2$, so by the $n = 1$ case, there is some $g_0 \in N_G(T)$ with $g_0 t_2 g_0^{-1} = t'_2$. Let $g_1 = g_0^{-1} g$; then $g_1 t_2 g_1^{-1} = g_0^{-1} g t_2 g^{-1} g_0 = g_0^{-1} t'_2 g_0 = t_2$, so $g_1$ is in the centralizer $Z_G(t_2)$. The centralizer is a reductive group with maximal torus $T$. Let $t''_1 = g_0^{-1} t'_1 g_0 = g_1 t_1 g_1^{-1}$. As the centralizer is a reductive group (although not necessarily connected), we can apply the $n = 1$ case again to get some element $g_2 \in N_{Z_G(t_2)}(T)$ with $g_2 t_1 g_2^{-1} = t''_1$. But through some rearrangement of the definition, 

$$N_{Z_G(t_2)}(T) = \{n \in Z_G(t_2) | n T n^{-1} = T\} = \{n \in G | n t_2 n^{-1} = t_2, n T n^{-1} = T\}$$

$$= N_G(T) \cap Z_G(t_2)$$

Let $g' = g_0 g_2$; an easy calculation shows that $g' (t_1, t_2) g'^{-1} = (t'_1, t'_2)$, and as $g_0, g_2$ are both in $N_G(T)$, the lemma is proven. 
\end{proof}

Define $X' = \{g | \forall i \, g t_i g^{-1} = t_i\} \subset X$; then $N_G(T) \cap Z_G(t_1, t_2, ..., t_n)$ acts on $X'$. There is an obvious map $X'/(N_G(T) \cap Z_G(t_1, t_2, ..., t_n)) \rightarrow X/N_G(T)$. Lemma \ref{Orbits lemma} allows us to construct an inverse map, as it implies that any element of $X/N_G(T)$ has some representative in $X'$, so the two are isomorphic. 

Therefore, we can rewrite $f^{-1}(t_1, t_2, ..., t_n) = \{(g, t_1, t_2, ..., t_n) | \forall i \linebreak g t_i g^{-1} = t_i\}/(N_G(T) \cap Z_G(t_1, t_2, ..., t_n))$. As the $n$-tuple in the numerator is now constant, this is isomorphic to $Z_G(t_1, t_2, ..., t_n)/N_{Z_G(t_1, t_2, ..., t_n)}(T)$. 

We now have that for each $x \in X$, the fiber is isomorphic to the quotient of a reductive group by the normalizer of its maximal torus. By the same trick as in the beginning of lemma \ref{Orbits lemma}, this is isomorphic to the quotient of a connected reductive group (the connected component of the identify of the original group) by the normalizer of its maximal torus. This is exactly the situation referred to in Theorem \ref{n=0} - so the fiber has trivial rational cohomology. This proves the theorem. 


\begin{thebibliography}{10}
\bibitem{ARXIVPAPER}
\label{ARXIVPAPER}Brion, M.: Equivariant cohomology and equivariant intersection theory. arXiv:math/9802063

\end{thebibliography}
\end{document}